\documentclass[12pt,a4paper]{article}
\usepackage{hyperref}
\usepackage{csquotes}
\usepackage{lipsum}
\usepackage[bbgreekl]{mathbbol}
\usepackage{amsfonts}
\usepackage{amsmath}
\usepackage{amssymb}
\DeclareSymbolFontAlphabet{\mathbb}{AMSb}
\DeclareSymbolFontAlphabet{\mathbbl}{bbold}

\usepackage[greek,english,ngerman]{babel}
\usepackage{graphicx}
\usepackage{amsmath}
\usepackage{verbatim}
\usepackage{color}
\usepackage{amstext}
\usepackage{latexsym}
\usepackage[small,it]{caption}
\usepackage{dsfont}
\usepackage{enumerate,enumitem}
\usepackage{mathtools}
\usepackage{indentfirst}
\usepackage{subcaption}
\usepackage{eurosym}
\usepackage{stmaryrd}

\usepackage[h]{esvect} % long vectors over things

\usepackage{bm}
\usepackage{amsthm}

\usepackage{url}

\makeatletter
\def\@endtheorem{\endtrivlist}% NEW
\makeatother
\theoremstyle{plain}
\newtheorem*{theorem*}{\protect\TheoremName}
\newtheorem*{acknowledgement*}{\protect\AcknowledgmentName}
\newtheorem*{algorithm*}{\protect\AlgorithmName}
\newtheorem*{assumption*}{\protect\AssumptionName}
\newtheorem*{assumptions*}{\protect\AssumptionsName}
\newtheorem*{axiom*}{\protect\AxiomName}
\newtheorem*{condition*}{\protect\ConditionName}
\newtheorem*{conditions*}{\protect\ConditionsName}
\newtheorem*{case*}{\protect\CaseName}
\newtheorem*{claim*}{\protect\ClaimName}
\newtheorem*{conclusion*}{\protect\ConclusionName}
\newtheorem*{conjecture*}{\protect\ConjectureName}
\newtheorem*{corollary*}{\protect\CorollaryName}
\newtheorem*{criterion*}{\protect\CriterionName}
\newtheorem*{definition*}{\protect\DefinitionName}
\newtheorem*{lemma*}{\protect\LemmaName}
\newtheorem*{notation*}{\protect\NotationName}
\newtheorem*{problem*}{\protect\ProblemName}
\newtheorem*{proposition*}{\protect\PropositionName}
\newtheorem*{solution*}{\protect\SolutionName}
\newtheorem*{summary*}{\protect\SummaryName}

\newtheorem{theorem}{\protect\TheoremName}[section]

\newtheorem{definition}[theorem]{\protect\DefinitionName}
\newtheorem{lemma}[theorem]{\protect\LemmaName}

\newtheorem{proposition}[theorem]{\protect\PropositionName}

\theoremstyle{definition}
\newtheorem*{example*}{\protect\ExampleName}
\newtheorem*{exercise*}{\protect\ExerciseName}
\newtheorem*{remark*}{\protect\RemarkName}

\newtheorem{remark}[theorem]{\protect\RemarkName}
\newtheorem*{question*}{\protect\QuestionName}
\newtheorem*{questions*}{\protect\QuestionsName}

\newcommand{\TheoremName}{}
\newcommand{\AcknowledgmentName}{}
\newcommand{\AlgorithmName}{}
\newcommand{\AssumptionName}{}
\newcommand{\AssumptionsName}{}
\newcommand{\QuestionName}{}
\newcommand{\QuestionsName}{}
\newcommand{\AxiomName}{}
\newcommand{\ConditionName}{}
\newcommand{\ConditionsName}{}
\newcommand{\CaseName}{}
\newcommand{\ClaimName}{}
\newcommand{\ConjectureName}{}
\newcommand{\ConclusionName}{}
\newcommand{\CorollaryName}{}
\newcommand{\CriterionName}{}
\newcommand{\DefinitionName}{}
\newcommand{\LemmaName}{}
\newcommand{\NotationName}{}
\newcommand{\ProblemName}{}
\newcommand{\PropositionName}{}
\newcommand{\SolutionName}{}
\newcommand{\SummaryName}{}
\newcommand{\ExampleName}{}
\newcommand{\RemarkName}{}
\newcommand{\ExerciseName}{}
\addto\captionsenglish{%
  \renewcommand{\TheoremName}{Theorem}
  \renewcommand{\AcknowledgmentName}{Acknowledgments}
  \renewcommand{\AlgorithmName}{Algorithm}
  \renewcommand{\AssumptionName}{Assumption}
  \renewcommand{\AssumptionsName}{Assumptions}
  \renewcommand{\QuestionName}{Question}
  \renewcommand{\QuestionsName}{Questions}
  \renewcommand{\AxiomName}{Axiom}
  \renewcommand{\ConditionName}{Condition}
  \renewcommand{\ConditionsName}{Conditions}
  \renewcommand{\CaseName}{Case}
  \renewcommand{\ClaimName}{Claim}
  \renewcommand{\ConjectureName}{Conjecture}
  \renewcommand{\ConclusionName}{Conclusion}
  \renewcommand{\CorollaryName}{Corollary}
  \renewcommand{\CriterionName}{Criterion}
  \renewcommand{\DefinitionName}{Definition}
  \renewcommand{\LemmaName}{Lemma}
  \renewcommand{\NotationName}{Notation}
  \renewcommand{\ProblemName}{Problem}
  \renewcommand{\PropositionName}{Proposition}
  \renewcommand{\SolutionName}{Solution}
  \renewcommand{\SummaryName}{Summary}
  \renewcommand{\ExampleName}{Example}%
  \renewcommand{\ExerciseName}{Exercise}%
  \renewcommand{\RemarkName}{Remark}%
}
\addto\captionsngerman{%
  \renewcommand{\TheoremName}{Theorem}
  \renewcommand{\AcknowledgmentName}{Anerkennungen}
  \renewcommand{\AlgorithmName}{Algorithmus}
  \renewcommand{\AssumptionName}{Voraussetzung}
  \renewcommand{\AssumptionsName}{Voraussetzungen}
  \renewcommand{\AxiomName}{Axiom}
  \renewcommand{\ConditionName}{Bedingung}
  \renewcommand{\ConditionsName}{Bedingungen}
  \renewcommand{\CaseName}{Fall}
  \renewcommand{\ClaimName}{Behauptung}
  \renewcommand{\ConjectureName}{Vermutung}
  \renewcommand{\ConclusionName}{Folgerung}
  \renewcommand{\CorollaryName}{Korollar}
  \renewcommand{\CriterionName}{Kriterium}
  \renewcommand{\DefinitionName}{Definition}
  \renewcommand{\LemmaName}{Lemma}
  \renewcommand{\NotationName}{Notation}
  \renewcommand{\ProblemName}{Problem}
  \renewcommand{\PropositionName}{Satz}
  \renewcommand{\SolutionName}{L\"osung}
  \renewcommand{\SummaryName}{Zusammenfassung}
  \renewcommand{\ExampleName}{Beispiel}%
  \renewcommand{\ExerciseName}{\"Ubung}%
  \renewcommand{\RemarkName}{Bemerkung}%
}
\addto\captionsgreek{%
  \renewcommand{\TheoremName}{Je\'wrhma}
  \renewcommand{\AcknowledgmentName}{Euqarist\'ies}
  \renewcommand{\AlgorithmName}{Alg\'orijmos}
  \renewcommand{\AssumptionName}{Up\'ojesh}
  \renewcommand{\AssumptionsName}{Upoj\'eseis}
  \renewcommand{\AxiomName}{Ax\'iwma}
  \renewcommand{\ConditionName}{Pro\"up\'ojesh}
  \renewcommand{\ConditionsName}{Pro\"upoj\'eseis}
  \renewcommand{\CaseName}{Per\'iptwsh}
  \renewcommand{\ClaimName}{Isqurism\'os}
  \renewcommand{\ConjectureName}{Eikas\'ia}
  \renewcommand{\ConclusionName}{Sump\'erasma}
  \renewcommand{\CorollaryName}{P\'orisma}
  \renewcommand{\CriterionName}{Krit\'hrio}
  \renewcommand{\DefinitionName}{Orism\'os}
  \renewcommand{\LemmaName}{L\'hmma}
  \renewcommand{\NotationName}{Shmeiograf\'ia}
  \renewcommand{\ProblemName}{Pr\'oblhma}
  \renewcommand{\PropositionName}{Pr\'otash}
  \renewcommand{\SolutionName}{L\'ush}
  \renewcommand{\SummaryName}{Per\'ilhyh}
  \renewcommand{\ExampleName}{Par\'adeigma}%
  \renewcommand{\ExerciseName}{\'Askhsh}%
  \renewcommand{\RemarkName}{Parat\'hrhsh}%
}
\newcommand{\itemEq}[1]{%
         \begingroup%
         \setlength{\abovedisplayskip}{0pt}%
         \setlength{\belowdisplayskip}{0pt}%
         \parbox[c]{\linewidth}{\begin{flalign}#1&&\end{flalign}}%
         \endgroup}

\makeatletter
\renewenvironment{proof}[1][\proofname]{\par
\pushQED{\hfill$\blacksquare$}%
\normalfont \topsep6\p@\@plus6\p@\relax
\trivlist
\item\relax
{\bfseries
#1\@addpunct{.}}\hspace\labelsep\ignorespaces
}{%
\popQED\endtrivlist\@endpefalse
}
\makeatother

\newcommand{\strongly}{%
  \mathrel{
    \resizebox{1.4em}{0.88ex}{$\rightarrow$}
}}

\newcommand{\weakly}{%
  \mathrel{
    \resizebox{1.4em}{0.88ex}{$\rightharpoonup$}
}}

\newcommand{\wstar}{%
  \mathrel{\vbox{\offinterlineskip\ialign{%
    \hfil##\hfil\cr
    $\scriptstyle*\,$\cr
    \noalign{\kern 0.12ex}
    \resizebox{1.4em}{0.88ex}{$\rightharpoonup$}\cr
}}}}

\newcommand{\twoweakly}{%
  \mathrel{\vbox{\offinterlineskip\ialign{%
    \hfil##\hfil\cr
    $\scriptstyle2\,$\cr
    \noalign{\kern 0.12ex}
    \resizebox{1.4em}{0.88ex}{$\rightharpoonup$}\cr
}}}}

\newcommand{\twostrongly}{%
  \mathrel{\vbox{\offinterlineskip\ialign{%
    \hfil##\hfil\cr
    $\scriptstyle2\,$\cr
    \noalign{\kern 0.12ex}
    \resizebox{1.4em}{0.88ex}{$\rightarrow$}\cr
}}}}

\def\XXint#1#2#3{{\setbox0=\hbox{$#1{#2#3}{\int}$ }\hspace{0.17em}
\vcenter{\hbox{$#2#3$ }}\kern-.585\wd0}}

\DeclareMathOperator*{\supp}{supp}

\DeclareMathOperator{\dist}{dist}

\newcommand{\lnorm}{\left\|}
\newcommand{\rnorm}{\right\|}


\makeatletter
\newcommand{\defeq}{\mathrel{\hspace{-0.1ex}\mathrel{\rlap{\raisebox{0.3ex}{$\m@th\cdot$}}\raisebox{-0.3ex}{$\m@th\cdot$}}\hspace{-0.73ex}\resizebox{0.55em}{0.84ex}{$=$}\hspace{0.67ex}}\hspace{-0.25em}}
\makeatother
\makeatletter
\newcommand{\eqdef}{\hspace{-0.25em}\mathrel{\hspace{0.67ex}\resizebox{0.55em}{0.84ex}{$=$}\hspace{-0.73ex}\mathrel{\rlap{\raisebox{0.3ex}{$\m@th\cdot$}}\raisebox{-0.3ex}{$\m@th\cdot$}}\hspace{-0.1ex}}}
\makeatother
\usepackage{tikz}


\usepackage[colorinlistoftodos]{todonotes}
\usepackage[author={A.S.},color=yellow]{pdfcomment}
\DTMsetregional %set date back to regional

\newcommand{\refpart}[2]{\hyperref[#1]{\ref*{#1}--\textit{#2.}}}
\newcommand{\N}{\mathbb{N}}
\newcommand{\R}{\mathbb{R}}
\newcommand{\C}{\mathbb{C}}
\begin{document}
\selectlanguage{english}

\title{Ground states for a nonlocal cubic-quartic Gross-Pitaevskii equation}
\author{Yongming Luo\thanks{Institut f\"{u}r Mathematik, Universit\"{a}t Kassel, 34132 Kassel, Germany}
\ and \
Athanasios Stylianou$^*$
}
\maketitle

\begin{abstract}
We prove existence and qualitative properties of ground state solutions to a generalized nonlocal 3rd-4th order Gross-Pitaevskii equation. Using a mountain pass argument on spheres and constructing appropriately localized Palais-Smale sequences we are able to prove existence of real positive ground states as saddle points. The analysis is deployed in the set of possible states, thus overcoming the problem that the energy is unbounded below.  We also prove a corresponding nonlocal Pohozaev identity with no rest term, a crucial part of the analysis.
\end{abstract}

\footnotetext[1]{\textbf{Keywords:} nonlocal mixed-order Gross-Pitaevskii equation, mountain pass on sphere}
\footnotetext[2]{\textbf{2010 AMS Subject Classification:} 35Q55, 49J35, 35B09}

\section{Introduction and main results}
We study existence of standing waves of the equation
\begin{equation}
 \label{GNLGPE} i\,\partial_t \psi=-\frac{1}{2}\,\Delta \psi+\lambda_1\,|\psi|^2\,\psi+\lambda_2\,(K*|\psi|^2)\,\psi+\lambda_3\,|\psi|^3\,\psi,\quad x\in\mathbb R^3,\ t>0,
\end{equation}
under the side constraint $\lnorm \psi(t)\rnorm_2^2=c$, as ground states of the corresponding energy. We assume that $\lambda_3<0$ and that $K$ is a convolution kernel of the form
\begin{equation*}
 \label{K} K(x)= \frac{x_1^2+x_2^2-2x_3^2}{|x|^5}.
\end{equation*}

This type of equations arises in the modelling of dipolar Bose-Einstein condensates. For $\lambda_3>0$, the equation models quantum fluctuations within the condensate, the so-called Lee-Huang-Yang correction; when the $\lambda_3$-term is of 5th order, the equation models three-body interactions (for more details and further references see \cite{LuoStylianou2019Pre} and references therein). The latter case is energy critical and its analysis is very complex due to the inherent loss of compactness for optimizing sequences. The energy critical case was studied in \cite{LuoStylianou2019Pre} for $\lambda_3>0$. Here we focus on the case $\lambda_3<0$ and restrict ourselves to the energy sub-critical case. This allows us to develop a method for obtaining saddle points on the constraint manifold, dealing first with the difficulty of the mixed order terms; the energy critical case will be a subject of future research. For this work, we use and appropriately modify the ideas from \cite{BellazziniJeanjean2016,Bellazzini2013}.

If we use the Fourier transform
\begin{equation*}
 \mathcal{F}(f)(\xi)=\widehat{f}(\xi)\defeq \int_{\R^3}f(x)e^{-ix\cdot \xi}\;dx
\end{equation*}
on $K$, we get
\begin{equation}\label{rangeKhat}
\widehat{K}(\xi)=\frac{4\pi}{3}\,\frac{2\xi_3^2-\xi_1^2-\xi_2^2}{|\xi|^2}\in\Big[-\frac{4}{3}\pi,\frac{8}{3}\pi\Big];
\end{equation}
see \cite[Lemma 2.3]{Carles2008}. Equation \eqref{GNLGPE} possesses a dynamically conserved energy functional, defined by
\begin{equation}
 \label{energy} E(u)\defeq \int_{\mathbb R^3}\Big\{\frac{1}{2}|\nabla u|^2+\frac{\lambda_1}{2}\,|u|^4+\frac{\lambda_2}{2}\,\big(K*|u|^2\big)\,|u|^2+\frac{2}{5}\,\lambda_3\,|u|^5\Big\}\;dx,
\end{equation}
which with the help of Parseval's identity becomes
\begin{equation*}
 \label{energy1} E(u)=\frac{1}{2}\|\nabla u\|_2^2+ \frac{1}{2}\frac{1}{(2\pi)^3}\,\int_{\R^3}\big(\lambda_1+\lambda_2\,\widehat{K}(\xi)\big)\,\big|\widehat{|u|^2}(\xi)\big|^2\;d\xi+\frac{2}{5}\lambda_3\,\|u\|_5^5.\\
\end{equation*}
For an arbitrary $c>0$, we look for ground states of \eqref{energy}, that is, for functions $u\in H^1(\mathbb R;\mathbb C)$ such that $\lnorm u\rnorm^2_{2}=c$, that are critical points of $E$ and study their qualitative properties. Note that a ground or excited state of $E$ corresponds to standing waves for \eqref{GNLGPE} through the Ansatz $\psi(x,t)=e^{-i\,\beta\,t}\,u(x)$; $\beta$ denotes the so-called \textit{chemical potential}. After making the standing wave Ansatz in \eqref{GNLGPE}, the problem reduces into finding a function $u:\R^3\strongly\C$ satisfying the side constraint $\lnorm u\rnorm^2_{2}=c$ and a number $\beta\in\R$ such that $(u,\beta)$ satisfies the equation
\begin{equation}\label{solution}
-\frac{1}{2}\Delta u +\lambda_1|u|^2u+\lambda_2 (K*|u|^2)u+\lambda_3|u|^3u+\beta u=0.
\end{equation}
\begin{definition}
We call $(u,\beta)\in H^1(\R^3;\C)\times \R$ a solution to equation \eqref{solution}, if the latter is satisfied in $H^{-1}(\R^3;\C)$ (with no side constraints).
\end{definition}
Here, the number $c>0$ denotes the mass of the solution. The rescaling we used (the same as in \cite{BellazziniJeanjean2016}) is such, that $c=1$ corresponds to the physical problem. The reason for studying the equation for a general $c>0$ is of technical nature and becomes apparent later in the paper.
\begin{definition}\label{definition of notations}
We will make extensive use of the following quantities:
\begin{align*}
A(u){}&\defeq \|\nabla u\|_2^2,\\[0.5em]
B(u){}&\defeq \frac{1}{(2\pi)^3}\,\int_{\R^3}\big(\lambda_1+\lambda_2\,\widehat{K}(\xi)\big)\,\big|\widehat{|u|^2}(\xi)\big|^2\;d\xi,\\[0.5em]
C(u){}&\defeq \lambda_3\|u\|_5^5,\\[0.5em]
Q(u){}& \defeq A(u)+\frac{3}{2}B(u)+\frac{9}{5}C(u),\\
\Xi {}&\defeq \frac{1}{(2\pi)^3}\max\bigg\{\Big|\lambda_1-\lambda_2\frac{4\pi}{3}\Big|,\Big|\lambda_1+\lambda_2\frac{8\pi}{3}\Big|\bigg\}.
\end{align*}
\end{definition}
\begin{remark}
 The virial functional $Q$ is closely related to Pohozaev identities: it is defined as such, so that critical points will satisfy $Q(u)=0$ (Lemma \ref{betaneq0}).
\end{remark}
\begin{remark}
Due to \eqref{rangeKhat}, we have $|\lambda_1+\lambda_2 \widehat{K}(\xi)|\leq \Xi$ for all $\xi\in\R^3$. This is an optimal inequality, since it becomes an equality (with plus or minus sign) for $\lambda_1$, $\lambda_2$ having the same sign and $\widehat K(\xi)=-4\pi/3$ or $\widehat K(\xi)=8\pi/3$. We thus have the following optimal estimate
\begin{equation}
 \label{B_estimate} |B(u)|\leq \Xi\,\lnorm u\rnorm_4^4, \text{ for all }\lambda_1,\lambda_2\in \mathbb R \text{ and } u\in L^2(\mathbb R^3)\cap L^4(\mathbb R^3).
\end{equation}
\end{remark}
\begin{remark}
Note that with the above definitions the following identity holds:
\begin{equation*}
E(u)=\frac{1}{2}A(u)+\frac{1}{2}B(u)+\frac{2}{5}C(u).
\end{equation*}
\end{remark}
\begin{definition}\label{definition of sets}
For a positive number $c>0$, the sets $S(c)$ and $V(c)$ are defined by
\begin{align*}
S(c)&\defeq\big\{u\in H^1(\R^3;\C):\|u\|_2^2=c\big\},\\
V(c)&\defeq \big\{u\in S(c):Q(u)=0\big\}.
\end{align*}
\end{definition}

Solutions to \eqref{solution} will be constructed as critical points of the energy $E$ in the constraint set $S(c)$ (For a more detailed exposition on the geometry of $S(c)$ as a Finsler manifold we refer to \cite{Berestycki1983} and references therein.)
\begin{remark}
 The energy functional $E$ is unbounded below on $S(c)$ for $\lambda_3<0$; see Lemma \refpart{infinity}{1}
\end{remark}
We want to study the existence of ground state solutions that appear as saddle points. To that end, in the spirit of \cite{BellazziniJeanjean2016,Bellazzini2013}, we give the following definitions:
\begin{definition}\label{ground_state}
 For an arbitrary  $c > 0$, we call $u_c \in S(c)$ a ground state, if it is a least-energy critical point on $S(c)$, i.e.,
 \begin{equation*}
  E(u_c)=\inf\Big\{E(u):u\in S(c)\text{ and } E|_{S(c)}'(u)=0\Big\},
 \end{equation*}
 where $E|_{S(c)}'(u)\in T^*_uS(c)$, i.e., $E|_{S(c)}':S(c)\strongly T^*S(c)$.
\end{definition}
\begin{remark}
 Note that the Lagrange multiplier theorem (see e.g. \cite[Corollary 3.5.29]{AbrahamEtAl1988}) implies that for any ground state $u$ exists $\beta\in\R$ such that $(u,\beta)$ is a solution.
\end{remark}
\begin{definition}[Mountain pass geometry]\label{def of MPG}Given $c > 0$, we say that $E$ has a mountain pass geometry on $S(c)$ at level $\gamma(c)\in \R$, if there exists $K_c > 0$, such that
\begin{equation}\label{MPG_prop}
\gamma(c)=\inf_{g\in\Gamma_c}\max_{t\in[0,1]}E\big(g(t)\big)>\sup_{g\in \Gamma_c}\max\{E\big(g(0)\big),E\big(g(1)\big)\},
\end{equation}
where
\begin{align*}
\Gamma_c \defeq  \big\{g\in C([0,1],S(c)): g(0)\in A_{K_c}\text{ and }E\big(g(1)\big)<0\big\}
\end{align*}
and
\begin{equation*}
 A_{K_c}\defeq \big\{u\in S(c):\|\nabla u\|_2^2\leq K_c\big\}.
\end{equation*}
\end{definition}

We point out that under certain general conditions, the existence of a Palais-Smale sequence is already guaranteed by the mountain pass geometry of the energy landscape, see, for instance, \cite[Theorem 4.1]{Ghoussoub1993}. However, boundedness of the sequence constructed in such way, can not be obtained directly in general. More precisely, the Palais-Smale sequence $\{u_n\}_{n\in\N}$ governed by \cite[Theorem 4.1]{Ghoussoub1993} implies the boundedness of the sequence $\big\{E(u_n)\big\}_{n\in\N}$. However, as long as the constant $\lambda_3$ is negative, the gradient energy $\|\nabla u_n\|_2^2$ can not be directly estimated from above by $E(u_n)$. In order to fix this problem, the pseudo-gradient technique introduced in \cite{Berestycki1983} and used, among others, by \cite{Bellazzini2013,BellazziniJeanjean2016}, will be utilized to obtain the sought for boundedness of the Palais-Smale sequence. We refer to Lemmas \ref{ps1} and \ref{ps2} for details.

Note that a ``self-bound'' crystal of droplets in a dilute Bose-Einstein condensate has been observed in \cite{KadauEtAl2016} where it was suggested that under specific circumstances it is a good candidate for a ground state. One year later, this suggestion was verified in \cite{WenzelEtAl2017}, so that, in contrast to \cite{BellazziniEtAl2017}, we do not expect that planar radial symmetry to be prominent in ground states. Still, they are positive, which conforms to the custom in  physics to pick ground states as the nodeless solutions. This is due to the fact that, as we will prove in the following,
\begin{equation}
 \label{virial_approach}\gamma(c)=\inf \big\{ E(u): u\in V(c)\big\}
\end{equation}
and that $Q(u)=0$ is a natural constraint for a critical point of $E$ on $S(c)$. The latter will be the unique mountain pass on the path $t\mapsto t^{3/2}\,u(t\,x)$ (and any such path will intersect the set over which the infimum is taken in the above).

As already stated, we cannot use standard tools from critical point theory, but we construct a special Palais-Smale sequence $\{u_n\}_{n\in\mathbb N}$ at level $\gamma(c)$, which concentrates around the set
$$V(c)\defeq \{u\in S(c)\text{ and } Q(u)=0\}.$$
To be more precise, we give the following:
\begin{definition}[$Q$-vanishing Palais-Smale sequence]
We call a sequence $\{u_n\}_{n\in\N}\subset S(c)$ a $Q$-vanishing Palais-Smale sequence at level $l\in\mathbb R$ for $E$ on $S(c)$, if the following hold:
\begin{enumerate}
\item $\displaystyle\lim_{n\to\infty}E(u_n)=l+o(1)$,
\item $\|E'|_{S(c)}(u_n)\|_{T^* S(c)}=o(1)$,
\item $Q(u_n)=o(1)$.
\end{enumerate}
\end{definition}

For such sequences, we are not able to apply concentration-compactness arguments, but proceed in the following way: We fist show that the weak limit will minimize $E$ in $V(c_1)$ for some $c_1\leq c$. We then use monotonicity properties of the mapping $c\mapsto \gamma (c)$ to show that $E(u_n-u)=o(1)$. This, in turn, will imply the strong convergence of the sequence and finally the fact that $E(u)=\gamma(c)$.

\begin{theorem}\label{Theorem1}
Let $c>0$ and $\lambda_3<0$.
\begin{enumerate}
\item There exists $K_c>0$ such that $E$ has mountain pass geometry on $S(c)$ at level $\gamma(c)>0$ (see Definition \ref{MPG}) and possesses no local minimizers on $S(c)$. Furthermore, the energy level $\gamma(c)$ is only determined by the value $c$ and is independent of the choice of $K_c$.
\item There exists $c_0>0$, depending only on $\lambda_1,\lambda_2,\lambda_3$, such that for all $c\in(0,c_0)$, equation \eqref{solution} possesses a nontrivial solution $(u_c,\beta_c)\in S(c)\times (0,\infty)$ such that $u_c$ is a ground state on $S(c)$ (in sense of Definition \ref{ground_state}) with
\begin{align*}
E(u_c)=\inf\Big\{E(u):u\in S(c)\text{ and } E|_{S(c)}'(u)=0\Big\}=\gamma(c).
\end{align*}
Moreover, if either
$$\lambda_2>0\ \text{ and }\ \lambda_1+\frac{8\pi}{3}\lambda_2\leq 0,$$
or
$$\lambda_2\leq 0\ \text{ and }\ \lambda_1-\frac{4\pi}{3}\lambda_2\leq 0$$
is satisfied, then $c_0=\infty$.
\item Let $(u,\beta)\in S(c)\times \R$ be a ground state. Then $(|u|,\beta)$ is also a ground state. In particular, $|u(x)|>0$ for all $x\in\R^3$ and there exists a constant $\theta\in\R$ such that $u=e^{i\theta}|u|$.
\end{enumerate}
\end{theorem}
The proof of the above theorem will be given in a number of lemmas and propositions in the following sections. In the end of the paper we will give a summary and conclude the argument.

\section{The energy landscape}
First we study the geometry of the energy landscape. This will allow us to construct a $Q$-vanishing Palais-Smale sequence in the following sections that converges to a ground state. To that end we will use the following scaling (see for example \cite{Cazenave2003})
\begin{equation}\label{Cazenave1}
 u^t(x)\defeq t^{3/2}\,u(tx)\ \text{ for }t>0,
\end{equation}
under which $S(c)$ is invariant. One calculates
\begin{equation}\label{Cazenave2}
 \begin{aligned}
  A(u^t)&=t^2\,A(u),\\
  B(u^t)&=t^3\,B(u),\\
  C(u^t)&=t^{9/2}\,C(u),
 \end{aligned}
\end{equation}
and therefore
\begin{align}
E(u^t)&=\frac{t^2}{2}A(u)+\frac{t^3}{2}B(u)+\frac{2}{5}t^{9/2}C(u),\label{e(u^t)}\\
Q(u^t)&=t^2A(u)+\frac{3t^3}{2}B(u)+\frac{9}{5}t^{9/2}C(u).\label{Q(u^t)}
\end{align}
\begin{remark}\label{unstable_regime}
Since $\lambda_3<0$, we infer that $C(u)<0$ for $u\neq 0$.
\end{remark}
In the following lemma we study the behaviour of the various expressions with respect to the above rescaling.
\begin{lemma}\label{infinity}
Let $c>0$, $\lambda_3<0$ and $u\in S(c)$. Then:
\begin{enumerate}
\item $A(u^t),B(u^t),C(u^t),E(u^t),Q(u^t)\strongly 0$ as $t\strongly 0$;\\[0.5em]
$A(u^t)\strongly \infty$ and $E(u^t)\strongly-\infty$ as $t\strongly\infty$.
\item If $E(u)<0$ then $Q(u)<0$.
\item There exists $k_0>0$ not depending on $u$ such that, if $A(u)\leq k_0$ then $Q(u)>0$ and $E(u)>0$.
\end{enumerate}
\end{lemma}
\begin{proof}
\textit{1.} Notice that for $u\neq 0$, we have $C(u)<0$ due to Remark \ref{unstable_regime}. Then due to the precise expression of the terms given by \eqref{Cazenave2} to \eqref{Q(u^t)} and the fact that the $C(u)$ term in \eqref{e(u^t)} and \eqref{Q(u^t)} is leading for large $t$, we obtain the assertion.

\textit{2.} From \eqref{e(u^t)} and \eqref{Q(u^t)} it follows that
\begin{equation*}
Q(u)-3E(u)=-\frac{1}{2}A(u)+\frac{3}{5}C(u)< 0\Longrightarrow Q(u)<3 E(u).
\end{equation*}
From this we obtain the second statement.

\textit{3.} We will use the following Gagliardo-Nirenberg inequalities:
\begin{align*}
\|u\|_5&\leq \mathrm C\,\|\nabla u\|_2^{9/10}\,\|u\|_2^{1/10}=\mathrm C\,c^{1/20}\, A(u)^{9/20},\\
\|u\|_4&\leq \mathrm C\,\|\nabla u\|_2^{3/4}\,\|u\|_2^{1/4}=\mathrm C\, c^{1/8}\, A(u)^{3/8},
\end{align*}
where $\mathrm C>0$ is a given positive constant independent of $u$. Recall from \eqref{rangeKhat} that
\begin{equation*}
\widehat{K}(\xi)=\frac{4\pi}{3}\,\frac{2\xi_3^2-\xi_1^2-\xi_2^2}{|\xi|^2}\in\Big[-\frac{4}{3}\pi,\frac{8}{3}\pi\Big]
\end{equation*}
and
from Definition \ref{definition of notations} that
\begin{equation*}
\Xi =\frac{1}{(2\pi)^3}\max\bigg\{\Big|\lambda_1-\lambda_2\frac{4\pi}{3}\Big|,\Big|\lambda_1+\lambda_2\frac{8\pi}{3}\Big|\bigg\}.
\end{equation*}
Therefore, since $\lambda_3< 0$ is assumed, we can estimate from below as follows:
\begin{align}
Q(u)&=A(u)+\frac{3}{2}B(u)+\frac{9}{5}C(u)\notag \\
&\geq A(u)-\frac{3}{2}\Xi \|u\|_4^4+\mathrm C\,\lambda_3\,c^{1/4}\, A(u)^{9/4}\notag \\
&\geq A(u)-\frac{3}{2}\Xi \mathrm C\, c^{1/2}\, A(u)^{3/2}+\mathrm C\,\lambda_3\,c^{1/4}\, A(u)^{9/4}\notag \\
&=A(u)-\mathrm C_1\,A(u)^{3/2}-\mathrm C_2\,A(u)^{9/4},\label{gagliardo-nirenberg and plancherel}
\end{align}
with positive constants $\mathrm C_1, \mathrm C_2$, since $\|u\|_2^2=c$ is constant. From the last inequality we see that $Q(u)>0$ for sufficiently small $A(u)$, say $A(u)\in(0,k_0)$ for some sufficiently small $k_0>0$ which does not depend on $u$. Analogously, using similar estimates as given by \eqref{gagliardo-nirenberg and plancherel} we also obtain that $E(u)>0$ for all $A(u)\in(0,k_0)$ by choosing the previous $k_0$ sufficiently small. This completes the proof.
\end{proof}
\begin{remark}\label{unbounded}
The previous lemma asserts that $E$ is unbounded below on $S(c)$.
\end{remark}
\begin{remark}\label{convenience}
From the proof of Lemma \ref{infinity} one can directly deduce that the $k_0$ given by \refpart{infinity}{3} can be replaced by an arbitrary $\hat{k}_0$ with $0<\hat{k}_0<k_0$. This property will be useful by proving the mountain pass geometry, see Proposition \ref{MPG} below.
\end{remark}
\begin{lemma}\label{monotoneproperty}
Let $c>0$, $\lambda_3<0$ and $u\in S(c)$. Then:
\begin{enumerate}
\item $\displaystyle\frac{\partial}{\partial t}E(u^t)=\frac{Q(u^t)}{t}$, for all $t>0$.
\item There exists a $t^*>0$ such that $u^{t^*}\in V(c)$.
\item We have $t^*(u)<1$ if and only if $Q(u)<0$. Moreover, $t^*(u)=1$ if and only if $Q(u)=0$.
\item The following inequalities hold:
\begin{equation*}
Q(u^t) \left\{
\begin{array}{lr}
             >0, &t\in(0,t^*(u)) ,\\
             <0, &t\in(t^*(u),\infty).
             \end{array}
\right.
\end{equation*}
\item $E(u^t)<E(u^{t^*})$ for all $t>0$ with $t\neq t^*$.
\end{enumerate}
\end{lemma}
\begin{proof}
Using \eqref{e(u^t)} and \eqref{Q(u^t)}, one directly verifies that
\begin{equation*}
\frac{\partial}{\partial t}E(u^t)=tA(u)+\frac{3t^2}{2}B(u)+\frac{9}{5}t^{7/2}C(u)=\frac{1}{t}Q(u^t).
\end{equation*}
This proves the first statement. Now define $y(t)\defeq \frac{\partial}{\partial t}E(u^t)$. Then
\begin{align*}
y'(t)&=A(u)+3tB(u)+\frac{63}{10}t^{5/2}C(u),\\
y''(t)&=3B(u)+\frac{63}{4}t^{3/2}C(u).
\end{align*}
If $B(u)\leq 0$, then $y''(t)$ is negative on $(0,\infty)$; If $B(u) >0$, then $y''(t)$ is positive on $(0,-\frac{4B(u)}{21C(u)})$ and negative on $(-\frac{4B(u)}{21C(u)},\infty)$. Since $y'(0)=A(u)>0$ and $y'(t)\to -\infty$ as $t\to \infty$, we conclude simultaneously from both cases that there exists a $t_0>0$ such that $y'(t)$ is positive on $(0,t_0)$ and negative on $(t_0,\infty)$. From the expression for $y(t)$ we obtain that $\lim_{t\searrow 0^+}y(t)=0$ and $\displaystyle \lim_{t\strongly\infty}y(t)=-\infty$. Thus $y(t)$ has a zero at $t^*>t_0$, $y(t)$ is positive on $(0,t^*)$ and $y(t)$ is negative on $(t^*,\infty)$. Since $y(t)=\frac{\partial E(u^t)}{\partial t}=\frac{Q(u^t)}{t}$, the second and the fourth statements are shown. The left statements are also direct consequences of the previous claims. This completes the proof.
\end{proof}

Having proved the above lemmas, we are now able to obtain the mountain pass geometry property of $E$ on $S(c)$:
\begin{proposition}\label{MPG}
Let $c>0$ and $\lambda_3<0$. Then there exists some $K_c>0$ such that the energy $E$ has a mountain pass geometry on $S(c)$ at level $\gamma(c)>0$.
\end{proposition}
\begin{proof}
We first define for $k>0$ the set
\begin{equation*}
C_k\defeq \{u\in S(c):A(u)=k\}
\end{equation*}
and the numbers
\begin{equation*}
 \alpha_k\defeq\sup_{u\in C_k}E(u)\ \text{ and }\ \beta_k\defeq\inf_{u\in C_k}E(u).
\end{equation*}
Note that $C_k\neq \emptyset$, since $A(u^t)=t^2\,A(u)$ for any $u\in S(c)$. We claim that:
\begin{equation}\label{k_1k_2}
 \begin{aligned}
  &\text{There exists }k_3>0\text{ such that for all } k_2\in (0,k_3]\text{ and }k_1\in(0,k_2)\\
  &\text{holds that }\alpha_k\leq\tfrac{1}{2}\beta_{k_2}\text{ for all }k\in[0,k_1].
 \end{aligned}
\end{equation}
\textit{Proof of the claim.} Let $k_2>0$ (to be determined). Estimating like \eqref{gagliardo-nirenberg and plancherel}, we obtain that there exist positive constants $\mathrm C_1, \mathrm C_2$ such that
$$E(u)\geq \frac{1}{2}A(u)-\mathrm C_1\,(A(u))^{3/2}-\mathrm C_2\,(A(u))^{9/4}.$$
We then define the real function
$$l(s)\defeq \frac{1}{2}s-\mathrm C_1\,s^{3/2}-\mathrm C_2\,s^{9/4}$$
for $s\in(0,\infty]$, which reads that $E(u)\geq l(A(u))$. Now let
$$g(s)\defeq l(s)-\frac{1}{4}s= \frac{s}{4}-\mathrm C_1\,s^{3/2}-\mathrm C_2\,s^{9/4},$$
then $g$ is positive for all sufficiently small positive $s$, say $s\in(0,k_3]$ for some sufficiently small $k_3>0$. This, in turn, implies that for $k_2\in (0,k_3]$ we have
$$E(u)\geq l(A(u))=l(k_2)\geq k_2/4$$
for all $u\in C_{k_2}$ and, therefore, $\beta_{k_2}\geq k_2/4$ for all $k_2\in(0,k_3]$. We pick a $k_2$ from $(0,k_3]$ and keep it fixed. Using Gagliardo-Nirenberg again as previously we obtain that there exist positive constants $\mathrm C_3,\mathrm C_4$ such that
\begin{align*}
E(u)&=\frac{1}{2}A(u)+\frac{1}{2}B(u)+\frac{2}{5}C(u)\notag \\
&\leq \frac{1}{2}A(u)+\mathrm C_3\,(A(u))^{3/2}+\mathrm C_4\,(A(u))^{9/4}\notag.
\end{align*}
Define
$$\hat{l}(s)\defeq \frac{1}{2}s+\mathrm C_3\,s^{3/2}+\mathrm C_4\,s^{9/4}$$
and
$$\hat{g}(s)\defeq \hat{l}(s)-s=- \frac{1}{2}s+\mathrm C_3\,s^{3/2}+\mathrm C_4\,s^{9/4},$$
then $\hat{g}(s)$ is negative for all sufficiently small positive $s$, say $s\in (0,\hat{k}_1]$ for some sufficiently small $\hat{k}_3>0$. This implies that
$$E(u)\leq \hat{l}(A(u))=\hat{l}(k)\leq k$$
for $u\in C_k$ and therefore, $\alpha_k\leq k$ for all $k\in (0,\hat{k}_1]$. Taking $k_1=\min\{\hat{k}_1,\frac{1}{8}k_2\}$, we finish the proof of claim \eqref{k_1k_2}.

Now, by construction of $k_2$, we see that $k_2$ can be replaced by an arbitrary $\hat{k}_2$ with $0<\hat{k}_2<k_2$. Thus we pick $k_2\leq k_0$, where $k_0$ is from Lemma \refpart{infinity}{3}. In order to apply certain contraposition argument below we also assume that $k_0=k_2$, which is valid due to Remark \ref{convenience}. We claim that taking $K_c=k_1$, where $k_1$ is given by \eqref{k_1k_2}, we are able to obtain \eqref{MPG_prop}, which completes the proof. Thus we let $\Gamma_c$ be given by
\begin{equation*}
\Gamma_c=\{g\in C([0,1],S(c)),g(0)\in A_{k_1},E\big(g(1)\big)<0\}.
\end{equation*}
First we show that $\Gamma_c\neq\emptyset$. Let $v\in S(c)$. Recall that $v^t(x)=t^{3/2}v(tx)$ and, in particular, $A(v^t)=t^2A(v)$. Therefore we can find a sufficiently small $t_1>0$ such that $A(v^{t_1})<k_1$. Moreover, from Lemma \refpart{infinity}{1} we can also pick a sufficiently large $t_2$ such that $E(v^{t_2})<0$. Now taking $g(t)\defeq v^{(1-t)t_1+t\,t_2}$, we see that $g$ is an element of $\Gamma_c$.

Thus let $g\in \Gamma_c$. Then $A\big(g(0)\big)\leq k_1<k_2$, which implies $Q\big(g(0)\big)>0$ (Lemma \refpart{infinity}{3}). Now since $E\big(g(1)\big)<0$, we infer from Lemma \refpart{infinity}{2} that $Q\big(g(1)\big)<0$, and therefore, by contraposition from Lemma \refpart{infinity}{3}, we have $A\big(g(1)\big)> k_2 $. Since $A\big(g(0)\big)<k_2$ and $A\big(g(1)\big)>k_2$, the continuity of $g$ implies that there exists a $t_0\in(0,1)$ such that $A\big(g(t_0)\big)=k_2$ and therefore $E\big(g(t_0)\big)\geq \beta_{k_2}$. Then,
\begin{align}
&\max_{t\in[0,1]} E\big(g(t)\big)\geq E\big(g(t_0)\big)\geq  \beta_{k_2}>\frac{1}{2}\beta_{k_2},\label{longineq1}\\
&\frac{1}{2}\beta_{k_2}\geq\alpha_{k_1}=\alpha_{A\big(g(0)\big)}\geq E\big(g(0)\big),\label{longineq2}\\
&E\big(g(0)\big)\geq\max\{E\big(g(0)\big),E\big(g(1)\big)\}\label{longineq}.
\end{align}
Here, \eqref{longineq1} and \eqref{longineq2} follow from the definitions of $\alpha_k$ and $\beta_k$ and the claim \eqref{k_1k_2}; \eqref{longineq} follows from the fact that $E\big(g(1)\big)<0$ and $E\big(g(0)\big)> 0$, since $E\big(g(0)\big)$ is positive for $g(0)\in A_{k_1}\subset A_{k_2}=A_{k_0}$ due to Lemma \refpart{infinity}{3} Finally, since the left-hand side of \eqref{longineq1} is bounded below by $\beta_{k_2}$ and the right-hand side of \eqref{longineq} is bounded above by $\frac{1}{2}\beta_{k_2}$, taking infimum and supremum over $g\in \Gamma_c$ in \eqref{longineq1} and \eqref{longineq} we obtain that
\begin{equation*}
\gamma(c)=\inf_{g\in\Gamma_c}\max_{t\in[0,1]}E\big(g(t)\big)>\sup_{g\in \Gamma_c}\max\{E\big(g(0)\big),E\big(g(1)\big)\},
\end{equation*}
which is exactly \eqref{MPG_prop} and this completes the desired proof.
\end{proof}

Next we show that $\gamma(c)$ is only determined by $c$ and is independent on the choice of $K_c$.
\begin{lemma}\label{infV(c)}
Let $c>0$ and $\lambda_3<0$. It holds that $\displaystyle \gamma(c)=\inf_{u\in V(c)}E(u)$.
\end{lemma}
\begin{proof}
Let $v\in V(c)$ so that $Q(v)=0$. Therefore from Lemma \refpart{monotoneproperty}{3} we conclude that $t^*(v)=1$. Due to equations \eqref{Cazenave2} and \eqref{e(u^t)}, we can find $0<t_1<1<t_2$ such that $v^{t_1}\in A_{k_1}$ and $E(v^{t_2})<0$. Now define
$$g(t)\defeq v^{(1-t)t_1+t\,t_2},$$
then $g\in\Gamma_c$. Using Lemma \refpart{monotoneproperty}{5} we get that
\begin{equation*}
\gamma(c)\leq \max_{t\in[0,1]}E\big(g(t)\big)=E(v^1)=E(v)
\end{equation*}
and therefore $\gamma(c)\leq\inf_{u\in V(c)}E(u)$. On the other hand, for a path $g\in \Gamma_c$, we obtain that $Q\big(g(0)\big)>0$ and $Q\big(g(1)\big)<0$ (as in the proof of Proposition \ref{MPG}). Thus due to continuity argument, any path in $\Gamma_c$ will cross $V(c)$. Hence
\begin{equation*}
\max_{t\in[0,1]}E\big(g(t)\big)\geq \inf_{u\in V(c)}E(u),
\end{equation*}
which implies that
$$\gamma(c)=\inf_{g\in\Gamma_c}\max_{t\in[0,1]}E\big(g(t)\big)\geq \inf_{u\in V(c)}E(u).$$
This completes the proof.
\end{proof}

Finally we prove the nonexistence of local minimizers.
\begin{proposition}\label{local minimizer}
Let $c>0$ and $\lambda_3<0$. Then the energy possesses no local minimizers on $S(c)$.
\end{proposition}
\begin{proof}
 Assume to the contrary that there exists a relatively open subset $A \subseteq S(c)$ and $v\in A$, such that
 \begin{equation*}
  E(v)=\inf\big\{E(w):w\in A\big\}.
 \end{equation*}
 Recall that $v^t(x)=t^\frac{3}{2}\,v(tx)$, so that $v^t\in A$ for all $t\in(1-\varepsilon,1+\varepsilon)$ for some $\varepsilon>0$ small enough. Then, since the mapping $t\mapsto E(v^t)$ has a local minimum at $t=1$, it must hold that $\partial_t\big(E(v^t)\big)\big|_{t=1}=0$ and $\partial_{tt}\big(E(v^t)\big)\big|_{t=1}\geq 0$. Recall that
\begin{align*}
\partial_t \big(E(v^t)\big)&=tA(v)+\frac{3}{2}t^2B(v)+\frac{9}{5}t^{\frac{7}{2}}C(v),\\
\partial_{tt} \big(E(v^t)\big)&=A(v)+3tB(v)+\frac{63}{10}t^{\frac{5}{2}}C(v).
\end{align*}
Evaluating at $t=1$ and then eliminating $B(v^t)$, we obtain that $-A(v)+\frac{27}{10}C(v)\geq 0$, a contradiction.
\end{proof}
\section{Construction and compactness of a $Q$-vanishing Palais-Smale-sequence}
In the following we use the idea given in \cite{Bellazzini2013} to construct a specific $Q$-vanishing Palais-Smale sequence. Let us consider the set
\begin{equation}\label{L}
L=\{u\in V(c):E(u)\leq \gamma(c)+1\}.
\end{equation}
The set $L$ is a bounded set in $H^1(\R^3;\C)$, since for $u\in L$ it follows $Q(u)=0$, and therefore
\begin{align*}
\gamma(c)+1&\geq E(u)=E(u)-\frac{1}{3}Q(u)=\frac{1}{6}A(u)-\frac{1}{5}C(u)\geq \frac{1}{6}A(u)=\frac{1}{6}\|\nabla u\|_2^2>0.
\end{align*}
Together with the fact that $u\in S(c)$ we obtain the boundedness of $L$. Now let $R_0>0$ be given such that $L\subset B(0,R_0)$, where $B(0,R_0)$ is the ball in $H^1(\R^3;\C)$ with center $0$ and radius $R_0$.
\begin{lemma}[Existence of a $Q$-vanishing Palais-Smale sequence]\label{ps1} 
Let $c>0$, $\lambda_3<0$ and
\begin{equation*}
J_{\mu}\defeq \big\{u\in S(c): |E(u)-\gamma(c)|\leq\mu,\ \dist\big(u,V(c)\big)\leq 2\mu\text{ and } \|E|_{S(c)}'(u)\|_{T^*S(c)}\leq 2\mu\big\}.
\end{equation*}
Then for any $\mu>0$, $J_{\mu}\cap B(0,3R_0)\neq\emptyset$.
\end{lemma}
\begin{proof}
Define the set
\begin{align*}
\Lambda_\mu&\defeq\big\{u\in S(c):|E(u)-\gamma(c)|\leq\mu,\ \dist\big(u,V(c)\big)\leq 2\mu\big\}.
\end{align*}
In view of \cite[Lemma 3.1]{Bellazzini2013}, we point out that we only need to show that for all sufficiently small $\varepsilon>0$, we can construct a path $g_{\varepsilon}\in\Gamma_c$ satisfying:
\begin{itemize}
 \item \itemEq{g_{\varepsilon}(t)=u^{(1-t)\theta_1+t\theta_2}\ \text{ for some } u\in V(c)\ \text{and}\ 0<\theta_1<1<\theta_2<\infty,\label{varepsilon0}}
 \item \itemEq{\max_{t\in[0,1]}E\big(g_{\varepsilon}(t)\big)\leq \gamma(c)+\varepsilon\ \text{ and}\label{varepsilon1}}
 \item \itemEq{E\big(g_{\varepsilon}(t)\big)\geq\gamma(c)\Longrightarrow g_{\varepsilon}(t)\in\Lambda_{\frac{\tilde\mu}{2}}\cap B(0,2R_0).\label{varepsilon2}}
\end{itemize}
We point out that the necessity of proving the existence of $g_{\varepsilon}$ ist due to the presence of the higher order term $\lambda_3|u|^3 u$, which makes our calculation differ from the ones given in \cite{Bellazzini2013}.

Now we give the precise construction of $g_{\varepsilon}$. Let $u\in V(c)$ with $E(u)\leq \gamma (c)+\varepsilon$ (which is valid, since $\gamma(c)=\inf_{u\in V(c)}E(u)$ due to Lemma \ref{infV(c)}) and let $0<\theta_1<1<\theta_2<\infty$ be chosen such that $u^{\lambda_1}\in A_{K_c},\ E(u^{\lambda_2})<0$ (see Lemma \refpart{infinity}{1}). We define $g_{\varepsilon}(t)$ by
\begin{equation*}
g_{\varepsilon}(t)\defeq u^{(1-t)\theta_1+t\theta_2}.
\end{equation*}
From Lemma \refpart{monotoneproperty}{3} and \refpart{monotoneproperty}{5} it follows that
\begin{equation}\label{eq3}
\max_{t\in[0,1]}E\big(g_{\varepsilon}(t)\big)\leq \gamma(c)+\varepsilon.
\end{equation}
Recall that
\begin{align*}
E(u^t)&=\frac{t^2}{2}A(u)+\frac{t^3}{2}B(u)+\frac{2}{5}t^{9/2}C(u),\\
Q(u^t)&=t^2A(u)+\frac{3t^3}{2}B(u)+\frac{9}{5}t^{9/2}C(u).
\end{align*}
Let $m(t)=(1-t)\theta_1+t\theta_2$. Basic calculus shows
\begin{align*}
\frac{d^2}{d t^2}E\big(g_{\varepsilon}(t)\big)&=(\theta_2-\theta_1)^2\,\Big(A(u)+3m(t)\,B(u)+\frac{63}{10}\,m(t)^{5/2}\,C(u)\Big),\\
\frac{d^3}{dt^3}E\big(g_{\varepsilon}(t)\big)&=(\theta_2-\theta_1)^3\,\Big(3B(u)+\frac{63}{4}m(t)^{3/2}\,C(u)\Big).
\end{align*}
Now let $t_{\varepsilon}\defeq \frac{1-\theta_1}{\theta_2-\theta_1}\in(0,1)$, so that $m(t_{\varepsilon})=1$. Then using $Q(u)=0$, we obtain that
\begin{align*}
\frac{d^2}{d t^2}E\big(g_{\varepsilon}(t)\big)\bigg|_{t=t_{\varepsilon}}&=(\theta_2-\theta_1)^2\,\Big(A(u)+3B(u)+\frac{63}{10}C(u)\Big)\notag \\
&=(\theta_2-\theta_1)^2\,\Big(2Q(u)-A(u)+\frac{27}{10}C(u)\Big)\notag \\
&=(\theta_2-\theta_1)^2\, \Big(-A(u)+\frac{27}{10}C(u)\Big)\eqdef(\theta_2-\theta_1)^2\, (-\zeta)<0,
\end{align*}
since $C(u)<0$.

Now let $t\in(0,1)$ with $E\big(g_{\varepsilon}(t)\big)\geq \gamma(c)$. We first consider the case $t=t_{\varepsilon}-h$ with $h>0$. Since $u\in V(c)$, Lemma \refpart{monotoneproperty}{1} implies $\frac{d}{dt}E(g_\varepsilon(k))\big|_{k=t_\varepsilon}=0$. Thus using Taylor expansion we see that there exists some $s\in [t,t_\varepsilon]$ such that
\begin{align}\label{gammac leq gammac+epsilon}
\gamma(c)\leq E\big(g_{\varepsilon}(t)\big)&=E\big(g_{\varepsilon}(t_\varepsilon)\big)+\frac{1}{2}(-h)^2 \frac{d^2}{d t^2}E\big(g_{\varepsilon}(t_{\varepsilon})\big)+\frac{1}{6}(-h)^3\frac{d^3}{dt^3}E\big(g_{\varepsilon}(s)\big)\notag \\
&\leq \gamma(c)+\varepsilon-\frac{h^2}{2}(\theta_2-\theta_1)^2\zeta-\frac{1}{6}h^3\frac{d^3}{dt^3}E\big(g_{\varepsilon}(s)\big).
\end{align}
Now since $h\in(0,t_\varepsilon)=\Big(0,\frac{1-\theta_1}{\theta_2-\theta_1}\Big)$, we have $m(s)\in [\theta_1,1]$ and we infer that
\begin{align*}
\bigg|-\frac{1}{6}h^3\frac{d^3}{dt^3}E\big(g_{\varepsilon}(s)\big)\bigg|&\leq\frac{1}{6}(1-\theta_1)^3\,\Big(3|B(u)|-\frac{63}{4}C(u)\Big)\notag \\
&\leq \frac{1}{6}\Big(3|B(u)|-\frac{63}{4}C(u)\Big)=:\tilde\zeta>0.
\end{align*}
From \eqref{gammac leq gammac+epsilon} it follows
\begin{align*}
h^2\leq\frac{2(\varepsilon+\tilde\zeta)}{(\theta_2-\theta_1)^2\zeta}.
\end{align*}
Since $\theta_2$ can be chosen arbitrary large, we pick a $\theta_2$ with $(\theta_2-\theta_1)^2\geq \frac{2(\varepsilon+\tilde\zeta)}{\varepsilon^2\zeta}$, thus
\begin{align*}
0< h\leq \varepsilon.
\end{align*}

Next we deal with the case $t=t_\varepsilon+h$ with $h>0$. From $Q(u)=0$ we obtain that
\begin{align*}
\frac{d^3}{dt^3}E\big(g_{\varepsilon}(s)\big)&=(\theta_2-\theta_1)^3\,\Big(3B(u)+\frac{63}{4}m(s)^{3/2}\,C(u)\Big)\\
&=(\theta_2-\theta_1)^3\,\Big(-2A-\frac{18}{5}C(u)+\frac{63}{4}m(s)^{3/2}\,C(u)\Big).
\end{align*}
From $s\in (t_{\varepsilon},t)$ it follows that $m(s)\in \big(m(t_{\varepsilon}),m(t)\big)=\big(1,m(t)\big)$. We hence obtain that
\begin{align}
\frac{d^3}{dt^3}E\big(g_{\varepsilon}(s)\big)&=(\theta_2-\theta_1)^3\,\Big(-2A-\frac{18}{5}C(u)+\frac{63}{4}m(s)^{3/2}\,C(u)\Big)\nonumber\\
&\leq (\theta_2-\theta_1)^3\,\Big(-2A-\frac{18}{5}C(u)+\frac{63}{4}\,C(u)\Big)\nonumber\\
&=(\theta_2-\theta_1)^3\,\Big(-2A+\frac{243}{20}\,C(u)\Big)<0\label{negative}
\end{align}
for $s\in [t_{\varepsilon},t]$, since $C(u)<0$. Then doing a Taylor expansion as in \eqref{gammac leq gammac+epsilon} (notice $-h$ in \eqref{gammac leq gammac+epsilon} is now replaced by $h$ and the third order term in \eqref{gammac leq gammac+epsilon} is negative due to \eqref{negative}) we obtain
\begin{align*}
\gamma(c)\leq E\big(g_{\varepsilon}(t)\big)\leq \gamma(c)+\varepsilon-\frac{h^2}{2}(\theta_2-\theta_1)^2\zeta.
\end{align*}
Thus if $(\theta_2-\theta_1)^2\geq\frac{2}{\varepsilon\zeta}$, then $0< h\leq \varepsilon$.

Therefore, we infer that picking $\theta_2$ with
\begin{equation}\label{eq2}
 (\theta_2-\theta_1)^2=\max\bigg\{\frac{2}{\varepsilon\zeta}, \frac{2(\varepsilon+\tilde\zeta)}{\varepsilon^2\zeta}\bigg\}
\end{equation}
implies
\begin{equation*}
\big\{t\in[0,1]:E\big(g_{\varepsilon}(t)\big)\geq \gamma (c)\big\}\subset (t_{\varepsilon}-\varepsilon,t_{\varepsilon}+\varepsilon).
\end{equation*}
Now, if $E\big(g_{\varepsilon}(t)\big)\geq \gamma (c)$, then \eqref{eq3} implies that $|E\big(g_{\varepsilon}(t)\big)-\gamma (c)|\leq \varepsilon$ and for $\varepsilon<\tilde{\mu}/2$ we get that $|E\big(g_{\varepsilon}(t)\big)-\gamma (c)|<\tilde{\mu}/2$.

Moreover, for $\varepsilon$ small enough, we get that $g_\varepsilon(t)\in L$, where $L$ is given in \eqref{L}. Thus $g_\varepsilon(t)\in B(0,2R_0)$.

Finally, from \eqref{eq2} and the fact that $\sqrt{\varepsilon}$ will dominate $\varepsilon$ for sufficiently small $\varepsilon$, we get that $(\theta_2-\theta_1)\leq \mathrm C/\sqrt{\varepsilon}$ for some positive constant $\mathrm C$. Let $v\in C_0^\infty(\R^3;\C)$ such that $\lnorm u-v\rnorm_{H^1}\leq\lambda_2\,\tilde{\mu}/4$. Since $v$ and its derivatives are Lipschitz continuous, there exists a constant $\mathrm C_{\tilde{\mu}}>0$ such that
\begin{equation*}
 \lnorm v^{t_1}-v^{t_2}\rnorm_{H^1}\leq \mathrm C_{\tilde{\mu}}\,|t_1-t_2|, \text{ for all } t_1,t_2\in[0,1].
\end{equation*}
Due to the fact that the $L^2$-norm is invariant with respect to the scaling \eqref{Cazenave1}, that the $L^2$-norm of the gradient rescales 1-homogeneously (see \eqref{Cazenave2}) and that $m(t_\varepsilon)=1$, we estimate
\begin{align*}
 \dist(g_\varepsilon(t),V(c))\leq{}& \lnorm g_\varepsilon(t)-g_\varepsilon(t_\varepsilon)\rnorm_{H^1}=\lnorm u^{m(t)}-u^{m(t_\varepsilon)}\rnorm_{H^1}\\
 \leq {}&\lnorm u^{m(t)}-v^{m(t)}\rnorm_{H^1}+\lnorm v^{m(t)}-v^{m(t_\varepsilon)}\rnorm_{H^1}+\lnorm u^{m(t_\varepsilon)}-v^{m(t_\varepsilon)}\rnorm_{H^1}\\
 \leq{}& 2\max\{1,m(t)\} \lnorm u-v\rnorm_{H^1}+\mathrm C_{\tilde{\mu}}\,|t-t_\varepsilon|\,|\lambda_2-\lambda_1|\\
 \leq {}& \frac{\tilde{\mu}}{2}+\mathrm C\,\mathrm C_{\tilde{\mu}}\,\sqrt{\varepsilon}\leq \tilde{\mu}
\end{align*}
for $\displaystyle\varepsilon\leq \frac{1}{4}\Big(\frac{\tilde{\mu}}{\mathrm C\,\mathrm C_{\tilde{\mu}}}\Big)^2$. All in all, we have shown \eqref{varepsilon2}.
\end{proof}

From Lemma \ref{ps1} we know that one can find a sequence $\{u_\mu\}_{\mu>0}$, which is bounded and $u_\mu$ is in $J_\mu$ for all $\mu>0$. Taking a sequence $\mu_{n}\to 0$ for $n\to\infty$, we obtain immediately a bounded Palais-Smale sequence. We make this precise by the following proposition:
\begin{lemma}\label{ps2}
Let $c>0$ and $\lambda_3<0$. Then there exists a $H^1$-bounded $Q$-vanishing Palais-Smale sequence $\{u_n\}_{n\in N}$ in $S(c)$, whose weak limit is nonzero.
\end{lemma}
\begin{proof}
It is identical to the proofs of \cite[Lemma 3.2]{Bellazzini2013}.
\end{proof}

Together with a truncation argument and the $pqr$-Lemma (see \cite{lieb1986}) we can prove compactness and exclude vanishing:
\begin{proposition}[Compactness of the $Q$-vanishing Palais-Smale sequence]\label{ps3}
Let $c>0$ and $\lambda_3<0$. Let $\{u_n\}_{n\in\N}\subset S(c)$ be the bounded Palais-Smale sequence constructed in Lemma \ref{ps2}. Then there exist $0\neq u\in H^1(\R^3,\C)$, $\beta\in\mathbb R$, a (not relabeled) subsequence $\{u_n\}_{n\in\N}$ and a sequence $\{\beta_n\}\subset\R$ such that:
\begin{enumerate}
\item $u_n\weakly u$ in $H^1(\R^3;\C)$.
\item $\displaystyle \beta_n\strongly\beta$ in $\R$.
\item $-\frac{1}{2}\Delta u_n+\beta_n\, u_n +\lambda_1|u_n|^2\,u_n+\lambda_2 (K*|u_n|^2)\,u_n+\lambda_3|u_n|^3\,u_n\strongly 0$ in $H^{-1}(\R^3;\C)$.
\item $-\frac{1}{2}\Delta u_n+\beta\, u_n +\lambda_1|u_n|^2\,u_n+\lambda_2 (K*|u_n|^2)\,u_n+\lambda_3|u_n|^3\,u_n\strongly 0$ in $H^{-1}(\R^3;\C)$.
\item $-\frac{1}{2}\Delta u+\beta\, u +\lambda_1|u|^2\,u+\lambda_2 (K*|u|^2)\,u+\lambda_3|u|^3\,u= 0$ in $H^{-1}(\R^3;\C)$.
\end{enumerate}
\end{proposition}
\begin{proof}
We refer the proof to \cite[Lemma 4.1, Proposition 4.1]{Bellazzini2013}.
\end{proof}
\begin{remark}\label{u is solution}
In particular we see that $(u,\beta)$ is a solution of \eqref{solution}, thus $Q(u)=0$ due to Lemma \ref{betaneq0} below.
\end{remark}

\section{Pohozaev identities and positivity of \texorpdfstring{$\beta$}{$\beta$}}
In the following we show a Pohozaev identity result, from which we obtain that the chemical potential $\beta$ is positive for sufficiently small $c$. The positivity of $\beta$ is essential for proving that the limit solution $u$ given by Proposition \ref{ps3} will lie in the set $S(c)$, see the proof of Theorem \ref{Theorem1} below.
\begin{lemma}\label{betaneq0}
Let $\lambda_3<0$ and $(u,\beta)\in H^1(\R^3;\C)\times \R$ be a solution of \eqref{solution}. Then $Q(u)=0$. Moreover, there exists some $c_0>0$, depending only on $\lambda_1,\lambda_2,\lambda_3$, such that for all $c\in (0,c_0)$, if $u\in S(c)$, then $\beta>0$. Moreover, if either
$$\lambda_2>0\ \text{ and }\ \lambda_1+\frac{8\pi}{3}\lambda_2\leq 0$$
or
$$\lambda_2\leq 0\ \text{ and }\ \lambda_1-\frac{4\pi}{3}\lambda_2\leq 0$$
is satisfied, then $c_0=\infty$.
\end{lemma}
\begin{proof}
Testing \eqref{solution} with $\bar{u}$ and $x\cdot \nabla \bar{u}$ and integrate by parts over $\R^3$ respectively (we point out that in order to ensure integration by parts, we should first truncate the test functions on bounded smooth domains on $\R^3$, then increase the radius of the domain to infinity. Since these arguments are standard and classical, we refer to \cite{Cazenave2003} for details) we deduce that
\begin{align}
&\frac{1}{2}A(u)+B(u)+C(u)+\beta\|u\|^2_{2}=0,\label{eq1}\\
&\frac{1}{4}A(u)+\frac{3}{4}B(u)+\frac{3}{5}C(u)+\frac{3}{2}\beta \|u\|^2_2=\frac{1}{4}\frac{\lambda_2}{(2\pi)^3}\int_{\R^3}\Big(\sum_{j=1}^3\xi_j\,\partial_j \widehat{W}(\xi)\Big)\,\big|\widehat{|u^2|}(\xi)\big|^2\;d\xi,\label{eq11}
\end{align}
where the right-hand side rest term of \eqref{eq11} comes from \cite[Proposition 5.3]{DeLaire2012}. But a direct calculation yields that $\sum_{j=1}^3\xi_j\,\partial_j \widehat{W}(\xi)=0$. Thus
\begin{equation}\label{eq12}
\frac{1}{4}A(u)+\frac{3}{4}B(u)+\frac{3}{5}C(u)+\frac{3}{2}\beta \|u\|^2_2=0.
\end{equation}
Eliminating $\|u\|^2_2$ from \eqref{eq1} and \eqref{eq12} we get that
\begin{equation}\label{qu=0}
Q(u)=A(u)+\frac{3}{2}B(u)+\frac{9}{5}C(u)=0.
\end{equation}
If is left to show the existence of $c_0$ with the above mentioned properties. From H\"older's inequality we obtain that
\begin{align}\label{hoelder}
\|u\|_4\leq \|u\|_5^{5/6}\|u\|_2^{1/6}.
\end{align}
We then discuss two cases: $\|u\|_4\geq 1$ and $\|u\|_4<1$. For the first case, we obtain from \eqref{hoelder} that
\begin{align*}
\|u\|_4^4\leq \|u\|_4^6\leq\|u\|^5_5\|u\|_2=c^{1/2}\,\|u\|_5^5.
\end{align*}
Thus
\begin{align}
\frac{1}{2}B(u)+\frac{1}{5}C(u)&\leq \frac{\Xi}{2}\,\|u\|_4^4+\frac{\lambda_3}{5}\,\|u\|_5^5\nonumber\\
&\leq\big(\frac{c^{1/2}\Xi}{2}+\frac{\lambda_3}{5}\big)\,\|u\|_5^5\label{c_estimate_1}.
\end{align}
Letting $c\in(0,\frac{4\lambda_3^2}{25\Xi^2})$, we see that the last term of \eqref{c_estimate_1} is negative. Now eliminating $A(u)$ from \eqref{eq1} and \eqref{eq12}, we obtain that
\begin{align}\label{bu+cu}
2\beta\|u\|_2^2=-\big(\frac{1}{2}B(u)+\frac{1}{5}C(u)\big)>0
\end{align}
for $c\in(0,\frac{4\lambda_3^2}{25\Xi^2})$, which implies that $\beta>0$. Now we consider the case $\|u\|_4<1$. We discuss two cases: $B(u)\leq 0$ and $B(u)>0$. For the first case,  we obtain directly from \eqref{bu+cu} that $\beta>0$ for all $c$ in $(0,\infty)$. Thus we assume that $B(u)>0$ in the following. Recall the Gagliardo-Nirenberg inequalities
\begin{align*}
-C(u)&\leq -\lambda_3 \mathrm C_1\,\|\nabla u\|_2^{9/2}\,\|u\|_2^{1/2}=-\lambda_3\mathrm C_1\, \,c^{1/4}\, A(u)^{9/4},
\end{align*}
where $\mathrm C_1$ is the Gagliardo-Nirenberg constant depending only on space dimension. Thus we obtain from \eqref{qu=0} that
\begin{align*}
A(u)+\frac{9}{5}\lambda_3\mathrm C_1c^{1/4}A(u)^{9/4}\leq A(u)+\frac{9}{5}C(u)=-\frac{3}{2}B(u)<0,
\end{align*}
which implies that
\begin{align}\label{eq21}
A(u)\geq \mathrm C_2(-\lambda_3)^{-4/5}c^{-1/5},
\end{align}
where $\mathrm C_2:=(\frac{5}{9})^{4/5}\mathrm C_1^{-4/5}$. On the other hand, eliminating $C(u)$ from \eqref{eq1} and \eqref{eq12}, we obtain that
\begin{align}\label{eq22}
18\beta\|u\|_2^2=A(u)-3B(u)\geq A(u)-3\Xi\|u\|_4^4\geq A(u)-3\Xi,
\end{align}
since $\|u\|_4\leq 1$. Thus letting
$$ \mathrm C_2(-\lambda_3)^{-4/5}c^{-1/5}> 3\Xi\Leftrightarrow c<\frac{\mathrm C_2^{5}\lambda_3^{-4}}{243}\Xi^{-5},$$
we conclude from \eqref{eq21} and \eqref{eq22} that $\beta>0$. Then $c_0:=\min\{\frac{4\lambda_3^2}{25\Xi^2},\frac{\mathrm C_2^{5}\lambda_3^{-4}}{243}\Xi^{-5}\}$ satisfies the assumptions of the lemma. Now if either
$$\lambda_2>0\ \text{ and } \ \lambda_1+\frac{8\pi}{3}\lambda_2\leq 0$$
or
$$\lambda_2\leq 0\ \text{ and }\ \lambda_1-\frac{4\pi}{3}\lambda_2\leq 0$$
is satisfied, we obtain that $B(u)\leq 0$ for all $u\in S(c)$, thus from the previous proof we immediately see that $c_0=\infty$. This completes the proof of the claim.
\end{proof}

\section{Proof of Theorem \ref{Theorem1}}
Before we finally prove the Theorem \ref{Theorem1} we still need a couple of technical tools.
\begin{lemma}\label{continuity}
Let $a>0,\ b\in\R,\ c< 0$ and
\begin{equation*}
f(a,b,c)\defeq \max_{t>0}\{at^2+bt^3+ct^{9/2}\}
\end{equation*}
Then $f$ is continuous in $(0,\infty)\times \R\times (-\infty,0)$.
\end{lemma}
\begin{proof}
Let $g(a,b,c,t)=at^2-bt^3-ct^{9/2}$. Then
\begin{align*}
\partial_t g(a,b,c,t)&=2at-3bt^2-\frac{9}{2}t^{7/2},\\
\partial_{tt}g(a,b,c,t)&=2a-6bt-\frac{63}{4}t^{5/2}.
\end{align*}
Setting $A(u)=a,\, B(u)=-b,\, C(u)=-c$, we deduce from the proof of Lemma \ref{monotoneproperty} that for each $a_0>0,b_0\in\R,c_0<0$, there exists a unique $t_0>0$ on such that $\partial_t g(a_0,b_0,c_0,t_0)=0$ and $\partial_{tt}g(a_0,b_0,c_0,t_0)<0$, which implies that $f(a_0,b_0,c_0)=g(a_0,b_0,c_0,t_0)$. Then using the implicit function theorem argument as in \cite[Lemma 5.2]{Bellazzini2013}, we obtain the result.
\end{proof}
\begin{lemma}\label{nonincreasing}
The function $c\mapsto \gamma(c)$ is non increasing for $c>0$.
\begin{proof}
Let $0<c_1<c_2$. To show the claim it suffices to show that for arbitrary $\varepsilon>0$ we have
\begin{equation*}
\gamma(c_2)\leq \gamma(c_1)+\varepsilon.
\end{equation*}
For any $u_1\in V(c_1)$, one obtains from Lemma \hyperref[monotoneproperty]{\ref*{monotoneproperty}, \textit{1.}} that
\begin{equation*}
E(u_1)=\max_{t>0}E(u_1^t).
\end{equation*}
Moreover, from Lemma \ref{infV(c)} we can find a $u_1\in V(c_1)$ such that $E(u_1)\leq \gamma(c_1)+\varepsilon/2$. Let $\eta\in C^{\infty}_0(\R^3)$ be a cut-off function with $\eta(x)=1$ for $|x|\leq 1$, $\eta(x)=0$ for $|x|\geq 2$ and $\eta\in [0,1]$ for $|x|\in(1,2)$. For $\delta>0$, define
\begin{equation*}
\tilde{u}_{1,\delta}(x)\defeq \eta(\delta x)\cdot u_1(x).
\end{equation*}
Then $\tilde{u}_{1,\delta}\strongly u_1$ in $H^1(\R^3;\C)$ as $\delta\strongly 0$. Therefore,
\begin{align*}
&A(\tilde{u}_{1,\delta})\strongly A(u_1),\notag \\
&B(\tilde{u}_{1,\delta})\strongly B(u_1),\notag \\
&C(\tilde{u}_{1,\delta})\strongly C(u_1)
\end{align*}
as $\delta\strongly 0$. Using the continuity property provided by Lemma \ref{continuity} we conclude that there exists sufficiently small $\delta>0$ such that
\begin{align*}
\max_{t>0}E(\tilde{u}_{1,\delta}^t)&=\max_{t>0}\Big\{\frac{t^2}{2}A(\tilde{u}_{1,\delta})+\frac{t^3}{2}B(\tilde{u}_{1,\delta})+\frac{2}{5}t^{9/2}C(\tilde{u}_{1,\delta})\Big\}\notag \\
&\leq\max_{t>0}\Big\{\frac{t^2}{2}A(u_1)+\frac{t^3}{2}B(u_1)+\frac{2}{5}t^{9/2}C(u_1)\Big\}+\frac{\varepsilon}{4}\notag \\
&=\max_{t>0} E(u_1^t)+\frac{\varepsilon}{4}.
\end{align*}
Now let $v\in C_0^\infty(\R^3)$ with $\supp(v)\subset\R^3\backslash B(0,2/\delta)$ and define
\begin{equation*}
v_0\defeq \frac{c_2-\|\tilde{u}_{1,\delta}\|_2^2}{\|v\|_2^2}\,v.
\end{equation*}
We have $\|v_0\|_2^2=c_2-\|\tilde{u}_{1,\delta}\|_2^2$ and for $\lambda\in(0,1)$,
\begin{align*}
A(v^\lambda_0)&=\lambda^2A(v_0),\notag \\
B(v^\lambda_0)&=\lambda^3B(v_0),\notag \\
C(v^\lambda_0)&=\lambda^{9/2}C(v_0).
\end{align*}
Let $w_{\lambda}\defeq \tilde{u}_{1,\delta}+v^\lambda_0$. For sufficiently small $\delta$ we get that $\supp \tilde u_{1,\delta}\cap \supp v_0^\lambda=\emptyset$. Thus
\begin{align*}
\lnorm w_{\lambda}\rnorm_2&= \lnorm\tilde{u}_{1,\delta}\rnorm_2+\lnorm v^\lambda_0\rnorm_2,\\
A(w_{\lambda})&=A(\tilde{u}_{1,\delta})+A(v_0^\lambda),\notag \\
C(w_{\lambda})&=C(\tilde{u}_{1,\delta})+C(v_0^\lambda)
\end{align*}
and the first equation above implies that $w_\lambda\in S(c_2)$. From Plancherel's identity, H\"older's inequality and the boundedness of $\widehat{K}$ we also obtain
\begin{align*}
&\big|B(w_{\lambda})-B(\tilde{u}_{1,\delta})-B(v_0^\lambda)\big|\notag \\
&\qquad =\bigg|\frac{1}{2}\frac{1}{(2\pi)^3}\int_{\R^3}\big(\lambda_1+\widehat{K}(x)\big)\,\Big(\big|\mathcal{F}\big((\tilde{u}_{1,\delta}+v_0^\lambda)^2\big)\big|^2-\big|\mathcal{F}(\tilde{u}_{1,\delta}^2)\big|^2-\big|\mathcal{F}\big((v_0^\lambda)^2\big)\big|^2\Big)\;dx\bigg|\notag \\
&\qquad \leq \mathrm C\int_{\R^3}\Big(\big|\tilde{u}_{1,\delta}^2\,(v_0^\lambda)^2\big|+\big|\tilde{u}_{1,\delta}^3\,v_0^\lambda\big|+\big|\tilde{u}_{1,\delta}\,(v_0^\lambda)^3\big|\Big)\;dx\notag \\
&\qquad \leq \mathrm C\,\Big(\|\tilde{u}_{1,\delta}\|^2_4\|v_0^\lambda\|^2_4+\|\tilde{u}_{1,\delta}\|^3_{\frac{9}{2}}\|v_0^\lambda\|_3+\|\tilde{u}_{1,\delta}\|_3\|v_0^\lambda\|^3_{9/2}\Big)\notag \\
&\qquad =\mathrm C\,\Big(\lambda^{3/2}\|\tilde{u}_{1,\delta}\|^2_4\,\|v_0\|^2_4+\lambda^{1/2}\|\tilde{u}_{1,\delta}\|^3_{9/2}\,\|v_0\|_3+\lambda^{5/2}\|\tilde{u}_{1,\delta}\|_3\,\|v_0\|^3_{9/2}\Big)\strongly 0
\end{align*}
as $\lambda\strongly 0$.
To sum up,
\begin{align*}
A(w_\lambda)&\strongly A(\tilde{u}_{1,\delta}),\notag \\
B(w_\lambda)&\strongly B(\tilde{u}_{1,\delta}),\notag \\
C(w_\lambda)&\strongly C(\tilde{u}_{1,\delta})
\end{align*}
as $\lambda\strongly 0$. Again using Lemma \ref{continuity}, we get that
\begin{align*}
\max_{t>0}E(w_\lambda^t)\leq \max_{t>0}E(\tilde{u}_{1,\delta}^t)+\frac{\varepsilon}{4}
\end{align*}
for sufficiently small $\lambda>0$. Finally, we calculate
\begin{align*}
\gamma(c_2)&\leq \max_{t>0}E(w_\lambda^t)\leq \max_{t>0}E(\tilde{u}_{1,\delta}^t)+\frac{\varepsilon}{4}\leq \max_{t>0}E(u_1^t)+\frac{\varepsilon}{2}=E(u_1)+\frac{\varepsilon}{2}\leq \gamma(c_1)+\varepsilon,
\end{align*}
which completes the proof.
\end{proof}
\end{lemma}

\begin{proof}[Proof of Theorem \protect\ref{Theorem1}]\label{proof of theorem 1}
\textit{1.} The geometry of the energy landscape is is proved by Proposition \ref{MPG}, the independence of $\gamma(c)$ on $K_c$ is proved by Lemma \ref{infV(c)} and the nonexistence of local minimizers in Proposition \ref{local minimizer}.

\textit{2.} We now show the second statement of Theorem \ref{Theorem1}. Let $c_0$ be the constant given by Lemma \ref{betaneq0}. Fix an arbitrary $c\in(0,c_0)$ and let $\{(u_n,\beta_n)\}_{n\in\N}$ and $(u,\beta)$ be given by Proposition \ref{ps3}. We show that $(u_c,\beta_c)\defeq(u,\beta)$ satisfies the assertion.

Using the splitting property given by \cite[Theorem 1]{BrezisLieb1983} we obtain that
\begin{align*}
B(u_n-u)+B(u)&=B(u_n)+o(1),\\
C(u_n-u)+C(u)&=C(u_n)+o(1),\\
D(u_n-u)+D(u)&=D(u_n)+o(1),
\end{align*}
where $D(v)\defeq \|v\|_2^2$. For $a,b$ elements of an arbitrary Hilbert space $H$, we have
\begin{align*}
( a-b,a-b)_H+( b,b)_H=( a,a)_H+( b,b-a)_H+( b-a,b)_H.
\end{align*}
If $a=b_n\weakly b$, it follows that
\begin{align*}
( b_n-b,b_n-b)_H+( b,b)_H=( b_n,b_n)_H+o(1).
\end{align*}
Thus we also have
\begin{align*}
A(u_n-u)+A(u)=A(u_n)+o(1).
\end{align*}
Since $E(u)=\frac{1}{2}A(u)+\frac{1}{2}B(u)+\frac{2}{5}C(u)$, we obtain that
\begin{equation*}
E(u_n-u)+E(u)=E(u_n)+o(1).
\end{equation*}
From Remark \ref{u is solution} we infer that $Q(u)=0$. From the lower semicontinuity of the $L^2$-norm we obtain that
$$\|u\|_2^2\leq \liminf_{n\to\infty}\|u_n\|_2^2=c. $$
Since $u\neq 0$, we infer that $u\in V(c_1)$ for some $c_1\in(0,c]$. It holds then $\gamma(c_1)=\inf_{u\in V(c_1)}E(u)$ from Lemma \ref{infV(c)}. We also have $\gamma(c)=E(u_n)+o(1)$ from Lemma \ref{ps2}. Hence
\begin{equation}\label{Elesso1}
E(u_n-u)+\gamma(c_1)\leq\gamma(c)+o(1).
\end{equation}
On the other hand, recall that
\begin{equation}\label{EminusQ}
-Q(v)+3E(v)=\frac{1}{2}A(v)-\frac{3}{5}C(v)
\end{equation}
for all $v\in H^1(\R^3;\C)$. Using $Q(u)=0$ and $Q(u_n)=o(1)$ from Lemma \ref{ps2} we obtain that
\begin{equation*}
Q(u_n-u)=Q(u_n-u)+Q(u)=Q(u_n)+o(1)=o(1).
\end{equation*}
Inserting this into \eqref{EminusQ}, we conclude that $E(u_n-u)\geq o(1)$, since the right-hand side of \eqref{EminusQ} is always nonnegative. From Lemma \ref{nonincreasing} we know that $\gamma(c_1)\geq \gamma(c)$, therefore it follows from \eqref{Elesso1} that $E(u_n-u)\leq o(1)$. Thus $E(u_n-u)=o(1)$. Since $A(v)$ and $-C(v)$ are nonnegative, we obtain from $E(u_n-u)= o(1)$, $Q(u_n-u)=o(1)$ and \eqref{EminusQ} that
\begin{align*}
A(u_n-u)&=o(1),\notag \\
C(u_n-u)&=o(1).
\end{align*}
But $E(u_n-u)$ is a linear combination of $A(u_n-u)$, $B(u_n-u)$ and $C(u_n-u)$, it follows immediately that $B(u_n-u)=o(1)$. Now the fourth and fifth statement of Proposition \ref{ps3} imply
\begin{equation*}
\frac{1}{2}A(u_n)+\beta D(u_n)+B(u_n)+C(u_n)=\frac{1}{2}A(u)+\beta D(u)+B(u)+C(u)+o(1).
\end{equation*}
Using the previous splitting properties one has then
\begin{align}
\beta D(u_n)&=\beta D(u)+o(1)\notag \\
&=\beta\big(D(u_n-u)+D(u)+o(1)\big).
\end{align}
From this we infer that $\beta D(u_n-u)=o(1)$. But from Lemma \ref{betaneq0}, $\beta>0$ for all $c\in(0,c_0)$. Thus $D(u_n-u)=o(1)$. Together with $A(u_n-u)=o(1)$ implies that $u_n\rightarrow u$ in $H^1(\R^3;\C)$ and $u\in S(c)$. From the splitting property of $E$ and the fact that $E(u_n-u)=o(1),\ E(u_n)=\gamma(c)+o(1)$ conclude that $E(u)=\gamma(c)$. From Lemma \ref{betaneq0} we know that a solution $u\in S(c)$ of \eqref{solution} is an element of $V(c)$. From Lemma \ref{infV(c)} we conclude that if additionally the solution $u$ satisfies $E(u)=\gamma(c)$, then it must be a ground state. This completes the proof of the second statement.

\textit{3.} The third statement is a very classical well-known result in literature, see for instance \cite[Chapter 8]{Cazenave2003}. Here, following the lines of \cite{Bellazzini2013}, we provide an alternative proof, which is shorter and less involved than the one given in \cite{Cazenave2003}.

We point out that we only need to show that $|u_c|$ is also a solution of \eqref{solution}. Indeed, having this we conclude immediately from the strong maximum principle that $|u_c(x)|$ is positive for all $x\in\R^3$. From Lemma \ref{betaneq0} we know that $|u_c|$ is an element of $V(c)$. From diamagnetic inequality and Lemma \ref{infV(c)} we then must have
$$\gamma(c)\leq E(|u_c|)\leq E(u_c)=\gamma(c),$$
which implies that $A(|u_c|)=A(u_c)$, and the existence of a $\theta\in\R$ with $u_c=e^{i\theta}|u_c|$ follows from \cite[Theorem 5]{BellazziniEtAl2017}.

We proceed our proof in two steps: First we show that $Q(|u_c|)=0$ and $E(|u_c|)=\gamma(c)$, from which we conclude that $|u_c|$ is a critical point of $V(c)$ due to Lemma \ref{infV(c)}; then we show that every critical point of $V(c)$ is also a critical point of $S(c)$, which finishes the desired proof. From the diamagnetic inequality we obtain that $Q(|u_c|)\leq Q(u_c)=0$. From Lemma \refpart{monotoneproperty}{3} we can find a $t\in(0,1]$ such that $Q(|u_c|^t)=0$. From Lemma \ref{infV(c)}, the diamagnetic inequality and Lemma \refpart{monotoneproperty}{5} we thus have
$$ \gamma(c)\leq E(|u_c|^t)=E(|u_c^t|)\leq E(u_c^t)\leq E(u_c)=\gamma(c).$$
Thus all inequalities must hold as equalities and we conclude that $t=1$, which finishes the first step. Now let $u$ be a critical point in $V(c)$. From the Lagrange multiplier theorem we obtain that there exist numbers $\mu_1$ and $\mu_2$ such that
\begin{align*}
E'(u)-\mu_1\,Q'(u)-2\mu_2\, u=0,
\end{align*}
from which we deduce that
\begin{align}\label{equality for pohozaev}
(1-2\mu_1)\,(-\Delta u)+2(1-3\mu_1)\,\big(\lambda_1\,|u|^2+(K*|u|^2)\big)\,u+(2-9\mu_1)\,\lambda_3\,|u|^3u-2\mu_2\,u=0.
\end{align}
Analogous to the proof of Lemma \ref{betaneq0} we obtain the following Pohozaev type identities:
\begin{align*}
(1-2\mu_1)\,A(u)+2(1-3\mu_1)\,B(u)+(2-9\mu_1)\,C(u)-2\mu_2\,\|u\|^2_2=0,\\
\frac{1}{2}(1-2\mu_1)\,A(u)+\frac{3}{2}(1-3\mu_1)\,B(u)+\frac{3}{5}(2-9\mu_1)\,C(u)-3\mu_2\,\|u\|^2_2=0.
\end{align*}
Eliminating $\|u\|^2_2$ from the last two equations we obtain that
\begin{align}\label{1}
(1-2\mu_1)\,A(u)+\frac{3}{2}(1-3\mu_1)\,B(u)+\frac{9}{10}(2-9\mu_1)\,C(u)=0.
\end{align}
Now since $u\in V(c)$ we know that
\begin{align}\label{2}
Q(u)=A(u)+\frac{3}{2}B(u)+\frac{9}{5}C(u)=0.
\end{align}
Eliminating $B(u)$ from \eqref{1} and \eqref{2} we get
\begin{align*}
\mu_1\Big(A(u)-\frac{27}{10}C(u)\Big)=0.
\end{align*}
Since $A(u)-\frac{27}{10}C(u)$ is positive, we must have $\mu_1=0$, which shows that $u$ is a critical point of $E(u)$ on $S(c)$.
\end{proof}

% \bibliography{../bib}
% \bibliographystyle{plain}
\end{document}